\documentclass[12pt,a4paper]{article}
\usepackage[utf8]{inputenc}
\usepackage{amsmath}
\usepackage{amssymb}

\makeatletter

\newcommand{\lyxaddress}[1]{
\par {\raggedright #1
\vspace{1.4em}
\noindent\par}
}

\@ifundefined{date}{}{\date{}}

\usepackage{amsthm}
\usepackage{mathrsfs}

\newtheorem{theorem}{Theorem}
\newtheorem{proposition}[theorem]{Proposition}
\newtheorem{lemma}[theorem]{Lemma}
\newtheorem{corollary}[theorem]{Corollary}
\theoremstyle{remark}
\newtheorem{remark}[theorem]{Remark}
\newtheorem*{remark*}{Remark}

\newcommand{\spec}{\mathop\mathrm{spec}\nolimits}
\newcommand{\diag}{\mathop\mathrm{diag}\nolimits}

\renewcommand{\Re}{\mathop\mathrm{Re}\nolimits}

\makeatother

\begin{document}

\title{A family of explicitly diagonalizable weighted Hankel matrices generalizing
the Hilbert matrix}

\author{T.~Kalvoda$^{1}$, P.~\v{S}\v{t}ov\'\i\v{c}ek$^{2}$}

\maketitle

\lyxaddress{$^{1}$Department of Applied Mathematics, Faculty of Information
Technology, Czech Technical University in~Prague, Th\'akurova~9,
160~00 Praha, Czech Republic}

\lyxaddress{$^{2}$Department of Mathematics, Faculty of Nuclear Science, Czech
Technical University in~Prague, Trojanova 13, 120~00 Praha, Czech
Republic}
\begin{abstract}
\noindent A three-parameter family $B=B(a,b,c)$ of weighted Hankel
matrices is introduced with the entries
\[
B_{j,k}=\frac{\Gamma(j+k+a)}{\Gamma(j+k+b+c)}\,\sqrt{\frac{\Gamma(j+b)\Gamma(j+c)\Gamma(k+b)\Gamma(k+c)}{\Gamma(j+a)\, j!\,\Gamma(k+a)\, k!}}\,,
\]
$j,k\in\mathbb{Z}_{+}$, supposing $a$, $b$, $c$ are positive and
$a<b+c$, $b<a+c$, $c\leq a+b$. The famous Hilbert matrix is included
as a particular case. The direct sum $B(a,b,c)\oplus B(a+1,b+1,c)$
is shown to commute with a discrete analog of the dilatation operator.
It follows that there exists a three-parameter family of real symmetric
Jacobi matrices, $T(a,b,c)$, commuting with \textbf{$B(a,b,c)$}.
The orthogonal polynomials associated with $T(a,b,c)$ turn out to
be the continuous dual Hahn polynomials. Consequently, a unitary mapping
$U$ diagonalizing $T(a,b,c)$ can be constructed explicitly. At the
same time, $U$ diagonalizes $B(a,b,c)$ and the spectrum of this
matrix operator is shown to be purely absolutely continuous and filling
the interval $[0,M(a,b,c)]$ where $M(a,b,c)$ is known explicitly.
If the assumption $c\leq a+b$ is relaxed while the remaining inequalities
on $a$, $b$, $c$ are all supposed to be valid, the spectrum contains
also a finite discrete part lying above the threshold $M(a,b,c)$.
Again, all eigenvalues and eigenvectors are described explicitly.
\end{abstract}
\vskip\baselineskip\noindent\emph{ Keywords}: weighted Hankel matrix;
diagonalization; spectrum; Hilbert's matrix

\vskip\baselineskip\noindent\emph{ AMS Subject Classification}:
47B35; 47B37; 47A10; 33C45

\section{Introduction}

An integral operator $K$ on $L^{2}((0,\infty),\mbox{d}x)$ whose
integral kernel $\mathcal{K}(x,y)$ is real, symmetric and homogeneous
of degree $-1$ on the first quadrant and such that
\[
\int_{0}^{\infty}\left|\mathcal{K}(t,1)\right|t^{-1/2}\mbox{d}t<\infty
\]
is bounded and explicitly diagonalizable by the Mellin integral transform.
In more detail, $K$ is unitarily equivalent to the multiplication
operator by the function
\[
g(\xi)=\int_{0}^{\infty}\mathcal{K}(t,1)\, t^{-1/2-i\xi}\mbox{d}t=\int_{\mathbb{R}}e^{-i\xi x}\mathcal{K}(e^{x/2},e^{-x/2})\,\mbox{d}x
\]
acting on $L^{2}(\mathbb{R},\mbox{d}\xi)$. This feature can readily
be understood if the symmetry properties of such an integral operator
are examined. $K$ commutes with the one-parameter unitary group of
dilatation transformations on the positive half-line generated by
the skew-symmetric differential operator $D=x\mbox{d}/\mbox{d}x+1/2$.
The kernel of the Mellin integral transform is in fact nothing but
a family of generalized eigenfunctions of $D$.

All what has been said above is applicable to the integral kernel
\begin{equation}
\mathcal{K}_{\ell}(x,y)=\frac{(xy)^{\ell/2}}{(x+y)^{\ell+1}}\label{eq:K_ell}
\end{equation}
depending on a real parameter $\ell$ provided $\ell>-1$. The corresponding
integral operator $K_{\ell}$ is unitarily equivalent to the multiplication
operator by the function
\begin{equation}
g(\xi)=\int_{0}^{\infty}t^{(\ell-1)/2+i\xi}(t+1)^{-\ell-1}\mbox{d}t=\frac{1}{\Gamma(\ell+1)}\left|\Gamma\!\left(\frac{1}{2}(\ell+1)+i\xi\right)\right|^{2}.\label{eq:fce_g}
\end{equation}

In no way it is straightforward to find an authentic discrete analog
of the integral kernel (\ref{eq:K_ell}). Of course, given a homogenous
kernel of degree $-1$ one can always restrict the kernel to the discrete
set $(\theta+\mathbb{Z}_{+})\times(\theta+\mathbb{Z}_{+})$, for some
$\theta>0$ and with $\mathbb{Z}_{+}$ standing for nonnegative integers,
obtaining this way a semi-infinite matrix. For example, using the
kernel (\ref{eq:K_ell}), with $\ell=0$, we get the (generalized)
Hilbert matrix. Matrices of this type have been explored, for instance,
in \cite{Kato58}. But as emphasized in the introduction of the cited
paper, it appears that there may be inconveniences in applying to
matrices some methods originally invented for integral operators.
In particular, let us note that no obvious discrete analog of the
Mellin integral transform yielding a diagonalization of such matrix
operators is at our disposal. Nevertheless, comparison of a matrix
operator obtained by discretizing an integral kernel with the integral
operator in question can provide, under certain additional assumptions,
quite a useful information including the precise value of the norm,
see Chp.~1 in \cite{Wilf}.

The present paper actually aims to describe matrix operators whose
character and properties remarkably closely resemble those of the
integral operators $K_{\ell}$. Notably, despite of nontrivial form
of such matrices all spectral properties are known in full detail.
Our approach is based on a construction of a three-parameter family
of matrices $B=B(a,b,c)$ with the desired structure of weighted Hankel
matrices and enjoying an appropriate symmetry. By the symmetry we
mean a discrete analog $D$ of the generator of dilatation transformations
which is commuting with $B$. $D$ can be regarded as a first order
difference operator. Then $D^{2}$ commutes with $B$ as well and,
since this is a second order difference operator, as such can be related
to a Jacobi (tridiagonal) matrix. Owing to the discrete symmetry $D$,
an explicit diagonalization of $B$ is possible. To this end, one
has to rely on the theory of orthogonal polynomials in place of the
Mellin transform. This is only roughly the main idea which is made
precise and fully developed in the remainder of the paper.

The famous generalized Hilbert matrix is included in the family $B(a,b,c)$
as a one-parameter subfamily. The constructed diagonalization procedure
if applied to Hilbert's matrix reproduces, in an alternative way,
a complete solution to the spectral problem due to Rosenblum \cite{RosenblumII}.
Considering only the original Hilbert matrix, i.e. not treating the
whole one-parameter family, it has been observed earlier by Otte that
a commuting Jacobi matrix can be used for a diagonalization. This
result exists in the form of a conference presentation which is currently
available from author's web page \cite{Otte}. But the first one who
has discovered that the generalized Hilbert matrix has a tridiagonal
matrix in its commutant seems to be Grünbaum in \cite{Grunbaum}.
As noted in the cited paper, the class of Hankel matrices with this
property is very limited. Saying this, the author was guided by his
forgoing studies of a similar problem for the class of Toeplitz matrices
\cite{Grunbaum81}.

\section{A family of weighted Hankel matrices with a discrete symmetry}

Consider the three-parameter family of semi-infinite symmetric matrices
$B=B(a,b,c)$,
\begin{equation}
B_{j,k}=\frac{\Gamma(j+k+a)}{\Gamma(j+k+b+c)}\sqrt{\frac{\Gamma(j+b)\Gamma(j+c)\Gamma(k+b)\Gamma(k+c)}{\Gamma(j+a)\, j!\,\Gamma(k+a)\, k!}},\ j,k\in\mathbb{Z}_{+}.\label{eq:Babc}
\end{equation}
The parameters $a$, $b$ and $c$ should be restricted to a range
for which the matrices $B(a,b,c)$ are real and hence Hermitian. Throughout
the paper we shall assume that the parameters $a$, $b$, $c$ are
all positive. The entries of $B$ have the structure $B_{j,k}=w(j)w(k)h(j+k)$
and therefore the matrices may be classified as weighted Hankel matrices,
see Chp. 6 $\S$8 in \cite{Peller}.

Note that, by the Stirling formula, the leading asymptotic term of
the matrix entry $B_{j,k}$ for both $j$ and $k$ large is
\[
B_{j,k}\sim\frac{(jk)^{(b+c-a-1)/2}}{(j+k)^{b+c-a}}.
\]
This may be compared with the integral kernel (\ref{eq:K_ell}), with
$\ell=b+c-a-1$, suggesting that we should suppose 
\begin{equation}
b+c-a>0.\label{eq:ineq}
\end{equation}

The matrix entries $B_{j,k}$ take particularly simple form if we
put $a=\theta$, $b=\theta+\ell$, $c=1$, with $\ell\in\mathbb{Z}_{+}$.
Then
\[
B_{j,k}=\frac{\sqrt{(j+\theta)_{\ell}(k+\theta)_{\ell}}}{(j+k+\theta)_{\ell+1}}
\]
where $(a)_{\ell}$ is the usual Pochhammer symbol. In particular,
for $\ell=0$ we obtain the generalized Hilbert matrix, $B_{j,k}=1/(j+k+\theta)$;
for a short account on its history see Chp. IX in \cite{HardyLittlewoodPolya}.

Moreover, letting $\ell=1$ we get a modification of the so called
Bergman-Hilbert matrix,
\[
B_{j,k}=\frac{\sqrt{(j+\theta)(k+\theta)}}{(j+k+\theta)(j+k+\theta+1)},
\]
studied in \cite{Ghatage,DavisGhatage}.

Put $A=B(a,b,c)\oplus B(a+1,b+1,c)$. $A$ can be identified with
the semi-infinite matrix with the entries
\[
A_{2j,2k}=B(a,b,c)_{j,k},\ A_{2j+1,2k+1}=B(a+1,b+1,c)_{j,k},\ j,k=0,1,2,\ldots,
\]
and $A_{j,k}=0$ if $j$, $k$ are of different parity. This identification
corresponds to the direct sum
\begin{equation}
\ell^{2}(\mathbb{Z}_{+})=\ell^{2}(2\mathbb{Z}_{+})\oplus\ell^{2}(2\mathbb{Z}_{+}+1).\label{eq:direct_sum}
\end{equation}
($\ell^{2}(2\mathbb{Z}_{+})$ is spanned by $\{e_{2k};\, k\in\mathbb{Z}_{+}\}$,
$\ell^{2}(2\mathbb{Z}_{+}+1)$ is spanned by $\{e_{2k+1};\, k\in\mathbb{Z}_{+}\}$
where $\{e_{k};\, k\in\mathbb{Z}_{+}\}$ is the standard basis in
$\ell^{2}(\mathbb{Z}_{+})$). Furthermore, let us introduce another
three-parameter family of matrices, $D=D(a,b,c)$, such that
\[
D_{j,j+1}=-D_{j+1,j}=d(j),\ D_{j,k}=0\ \text{otherwise},
\]
where
\[
d(2j)=\sqrt{(j+a)(j+b)},\ d(2j+1)=\sqrt{(j+1)(j+c)},\ j,k=0,1,2,\ldots.
\]
$D$ can be regarded as a discrete analog of the dilatation operator.

With respect to the decomposition (\ref{eq:direct_sum}), the matrices
$A$ and $D$ can be written in the blockwise form
\begin{equation}
A=\left(\begin{array}{cc}
B(a,b,c) & 0\\
0 & B(a+1,b+1,c)
\end{array}\right)\!,\ D=\left(\begin{array}{cc}
0 & C\\
-C^{T} & 0
\end{array}\right)\!,\label{eq:matAD}
\end{equation}
where $C$ is another semi-infinite matrix with the entries
\[
C_{j,j}=d(2j),\ C_{j+1,j}=-d(2j+1),\ C_{j,k}=0\ \text{otherwise}.
\]

\begin{lemma} $A$ and $D$ commute. \end{lemma}

\begin{proof} The proof can be carried out relying just on a straightforward
evaluation of the matrix entries of the commutator. One has (with
$d(-1):=0$)
\[
(AD-DA)_{j,k}=d(k-1)A_{j,k-1}-d(k)A_{j,k+1}+d(j-1)A_{j-1,k}-d(j)A_{j+1,k}.
\]
Obviously, $(AD-DA)_{j,k}$ vanishes whenever $j$ and $k$ are of
the same parity. Without loss of generality it suffices to consider
the situation when $j$ is even and $k$ is odd. Let us suppose $j,k\in\mathbb{Z}_{+}$
and evaluate the expression
\begin{eqnarray*}
(AD-DA)_{2j,2k+1} & = & d(2k)B(a,b,c)_{j,k}-d(2k+1)B(a,b,c)_{j,k+1}\\
 &  & +\, d(2j-1)B(a+1,b+1,c)_{j-1,k}-d(2j)B(a+1,b+1,c)_{j,k}.
\end{eqnarray*}
A direct computation actually shows that the expression equals zero.

One may prefer, however, to proceed another way revealing the algebraic
structure behind the identity. In view of (\ref{eq:matAD}), $AD-DA=0$
is equivalent to
\begin{equation}
B(a,b,c)\, C=CB(a+1,b+1,c),\label{eq:intertwinBC}
\end{equation}
i.e. $C$ intertwines the operators $B(a,b,c)$ and $B(a+1,b+1,c)$.
In order to take into account the structure of $B=B(a,b,c)$ let us
write $B=WHW$ where $W$ is the diagonal matrix with the diagonal
entries
\[
W_{j,j}=\sqrt{\frac{\Gamma(j+b)\Gamma(j+c)}{\Gamma(j+a)\, j!}},
\]
and $H$ is the Hankel matrix, $H_{j,k}=h(j+k)$, where
\[
h(z)=\frac{\Gamma(z+a)}{\Gamma(z+b+c)}.
\]
Similarly, $B(a+1,b+1,c)=\tilde{W}\tilde{H}\tilde{W}$, with analogous
expressions obtained just by shifting the parameters $a$ and $b$
by $1$. Note that $\tilde{h}(z)=h(z+1)$. Then (\ref{eq:intertwinBC})
can be rewritten as
\[
HV=\tilde{V}\tilde{H},\ \text{with}\ V=WC\tilde{W}^{-1},\ \tilde{V}=W^{-1}C\tilde{W}.
\]
One readily finds that
\[
V_{j,j}=j+a,\ V_{j+1,j}=-j-c,\ \tilde{V}_{j,j}=j+b,\ \tilde{V}_{j+1,j}=-j-1,\ V_{j,k}=\tilde{V}_{j,k}=0\,\ \text{otherwise}.
\]
Consequently, (\ref{eq:intertwinBC}) means that
\[
h(j+k)\,(j+k+a)=\tilde{h}(j+k)\,(j+k+b+c)
\]
 for all indices $j$, $k$. But this is actually so since
\[
\tilde{h}(z)=h(z+1)=\frac{z+a}{z+b+c}\, h(z),
\]
thus concluding the proof. \end{proof}

\begin{corollary} Let $T=T(a,b,c)$ be the three-parameter family
consisting of symmetric Jacobi matrices with the entries
\[
T_{j,j}=j\,(j+c-1)+(j+a)(j+b),\ T_{j,j+1}=T_{j+1,j}=-\sqrt{(j+1)(j+a)(j+b)(j+c)},
\]
$T_{j,k}=0$ otherwise. Then the matrices $B(a,b,c)$ and $T(a,b,c)$
commute. \end{corollary}

\begin{proof} Clearly, $A$ and $D^{2}$ commute and therefore, in
view of (\ref{eq:matAD}), $B$ and $T=CC^{T}$ commute. $T$ is a
Jacobi matrix with the entries
\begin{eqnarray*}
T_{j,j} & = & d(2j-1)^{2}+d(2j)^{2}\,=\, j\,(j+c-1)+(j+a)(j+b),\\
T_{j,j+1} & = & T_{j+1,j}\,=\,-d(2j)d(2j+1)\,=\,-\sqrt{(j+1)(j+a)(j+b)(j+c)},
\end{eqnarray*}
and $T_{j,k}=0$ otherwise. \end{proof}

\section{\label{sec:diagonalization} Diagonalization and the spectral properties
of $B(a,b,c)$}

\subsection{The associated orthogonal polynomials}

The monic orthogonal polynomials associated with a Jacobi matrix $T$
are defined by the recurrence: $P_{-1}(x)=0$, $P_{0}(x)=1$, and
\[
P_{j+1}(x)=\left(x-T_{j+1,j+1}\right)P_{j}(x)-\left(T_{j,j+1}\right)^{2}P_{j-1}(x).
\]
In our case, it is convenient to modify the matrix $T$ by adding
a multiple of the unit operator. The redefined diagonal entries read
\begin{eqnarray*}
T_{j,j} & = & j(j+c-1)+(j+a)(j+b)-\frac{1}{4}(a+b-c)^{2}\\
 & = & -\, j\,(j+1)+(j-b+d)(j-c+d)+(j-a+d)(j-b+d)\\
 &  & +\,(j-a+d)(j-c+d)
\end{eqnarray*}
where $d=(a+b+c)/2$. The associated monic orthogonal polynomials,
$P_{n}(x)$, coincide with the continuous dual Hahn polynomials; see
Eq. 9.3.5 in \cite{KoekoekLeskySwarttouw}. We have (see also \cite[Eq. 9.3.1]{KoekoekLeskySwarttouw})
\begin{eqnarray*}
P_{n}(x^{2}) & = & (-1)^{n}\, S_{n}\!\left(x^{2};\frac{b+c-a}{2},\frac{a+c-b}{2},\frac{a+b-c}{2}\right)\\
 & = & (-1)^{n}(b)_{n}(c)_{n}\,\,_{3}F_{2}\!\left(-n,\frac{b+c-a}{2}+ix,\frac{b+c-a}{2}-ix;b,c;1\right)\!.
\end{eqnarray*}

The measure of orthogonality for the continuous dual Hahn polynomials
is known explicitly \cite[Eqs. 9.3.2, 9.3.3]{KoekoekLeskySwarttouw}.
Put
\[
\rho(x)=\frac{x\sinh(2\pi x)}{\pi^{2}\Gamma(a)\Gamma(b)\Gamma(c)}\left|\Gamma\!\left(\frac{b+c-a}{2}+ix\right)\!\Gamma\!\left(\frac{a+c-b}{2}+ix\right)\!\Gamma\!\left(\frac{a+b-c}{2}+ix\right)\right|^{2}\!.
\]
Then, for $a$, $b$, $c$ positive such that $b+c>a$, $a+c>b$ and
$a+b\geq c$, we have the orthogonality relation
\[
\int_{0}^{\infty}P_{m}(x^{2})P_{n}(x^{2})\rho(x)\,\mbox{d}x=(a)_{n}(b)_{n}(c)_{n}\, n!\,\delta_{m,n}.
\]

More generally, suppose for definiteness that $0<b\leq c$ and, in
agreement with (\ref{eq:ineq}), $0<a<b+c$. Then obviously $a+c-b>0$
but it may happen that $a+b-c<0$ (then necessarily $b<c$). In that
case the orthogonality relation should be modified by adding a discrete
part. So suppose
\begin{equation}
0<b<c\ \,\text{and}\,\ 0<a<c-b\label{eq:ineq_dicrete}
\end{equation}
(then $a<c-b<b+c$). Put ($\lceil x\rceil$ meaning the ceiling of
$x\in\mathbb{R}$)
\begin{equation}
N(a,b,c)=\lceil(c-a-b)/2\rceil-1.\label{eq:N}
\end{equation}
One has
\begin{eqnarray}
 &  & \int_{0}^{\infty}P_{m}(x^{2})P_{n}(x^{2})\rho(x)\,\mbox{d}x+\frac{\Gamma(c-a)\Gamma(c-b)}{\Gamma(c)\Gamma(c-a-b)}\nonumber \\
 &  & \quad\times\sum_{k=0}^{N(a,b,c)}\left(1+\frac{2k}{a+b-c}\right)\frac{(-1)^{k}(a+b-c)_{k}(a)_{k}(b)_{k}}{(a-c+1)_{k}(b-c+1)_{k}\, k!}\label{eq:OGrel}\\
 &  & \qquad\qquad\times P_{m}\!\Bigg(\!-\!\left(\frac{a+b-c}{2}+k\right)^{\!2}\Bigg)P_{n}\!\Bigg(\!-\!\left(\frac{a+b-c}{2}+k\right)^{\!2}\Bigg)\nonumber \\
 &  & =(a)_{n}(b)_{n}(c)_{n}\, n!\,\delta_{m,n}.\nonumber 
\end{eqnarray}

\subsection{Diagonalization of $B(a,b,c)$}

Assuming (\ref{eq:ineq}) let
\begin{equation}
\hat{P}_{n}(x^{2})=S_{n}\!\left(x^{2};\frac{b+c-a}{2},\frac{a+c-b}{2},\frac{a+b-c}{2}\right)\!\bigg/\!\sqrt{(a)_{n}(b)_{n}(c)_{n}\, n!}.\label{eq:ONpoly}
\end{equation}
The polynomials $\hat{P}_{n}(x^{2})$ are normalized and $\big(\hat{P}_{0}(x^{2}),\hat{P}_{1}(x^{2}),\hat{P}_{2}(x^{2}),\ldots\big)$
is a formal eigenvector of the matrix operator $T(a,b,c)$ corresponding
to the eigenvalue $x^{2}$. Let us introduce the unitary transform
\begin{equation}
U:\ell^{2}(\mathbb{Z}_{+})\to L^{2}\big(\mathcal{M}(a,b,c),\mbox{d}\mu\big)\!:e_{n}\to\hat{P}_{n}(x^{2}),\ n\in\mathbb{Z}_{+},\label{eq:U}
\end{equation}
where
\begin{equation}
\mathcal{M}(a,b,c)=(0,+\infty)\cup\{\lambda_{k};\, k=0,1,\ldots,N(a,b,c)\},\ \lambda_{k}=i\left(\frac{a+b-c}{2}+k\right)\label{eq:calM}
\end{equation}
(the discrete part occurs if and only if $a+b-c<0$), $\mbox{d}\mu(x)=\rho(x)\mbox{d}x$
on $(0,+\infty)$ and
\[
\mu(\{\lambda_{k}\})=\frac{(-1)^{k}\Gamma(c-a)\Gamma(c-b)}{\Gamma(c)\Gamma(c-a-b)}\left(1+\frac{2k}{a+b-c}\right)\frac{(a+b-c)_{k}(a)_{k}(b)_{k}}{(a-c+1)_{k}(b-c+1)_{k}\, k!}.
\]

\begin{remark*} Note that if $0\leq k\leq N(a,b,c)$ and (\ref{eq:ineq_dicrete})
is true then $\mu(\{\lambda_{k}\})>0$ as it should be. In fact, the
signs of the numbers $(a+b-c)_{k}$, $(a-c+1)_{k}$, $(b-c+1)_{k}$
are all equal to $(-1)^{k}$. This is also a standard fact that $(\mathcal{M}(a,b,c),\mbox{d}\mu)$
is a probability space, as seen from (\ref{eq:OGrel}) (with $m=n=0$).
\end{remark*}

\begin{theorem} \label{thm:B_multop} The matrix operator $B(a,b,c)$
on $\ell^{2}(\mathbb{Z}_{+})$ is unitarily equivalent to the multiplication
operator by the function
\begin{equation}
h(x)=\frac{1}{\Gamma(b+c-a)}\left|\Gamma\!\left(\frac{b+c-a}{2}+ix\right)\right|^{2}\label{eq:h}
\end{equation}
acting on $L^{2}(\mathcal{M}(a,b,c),\mbox{d}\mu)$. \end{theorem}

\begin{remark*} Once more, this result should be compared to the
integral operator $K_{\ell}$ which is unitarily equivalent to the
multiplication operator by the function $g(\xi)$ introduced in (\ref{eq:fce_g})
and acting on $L^{2}(\mathbb{R},\mbox{d}\xi)$. Again, one has to
put $\ell=b+c-a-1$. \end{remark*}

\begin{proof} Directly from the construction it follows that $U$
diagonalizes $T$, namely $UTU^{-1}$ equals the multiplication operator
by $x^{2}$. Since $B$ and $T$ commute and the spectrum of $T$
is simple, $UBU^{-1}$ is necessarily a multiplication operator, too,
say by a function $h(x)$; see, for instance, Lemma 6.4 in \cite{Varadarajan}
or Proposition 1.9 in Supplement~1 of \cite{BerezinShubin}. One
derives
\begin{eqnarray*}
h(x) & = & h(x)\hat{P}_{0}(x^{2})=UBe_{0}=\sum_{j=0}^{\infty}B_{0,j}\hat{P}_{j}(x^{2})\\
 & = & \sum_{j=0}^{\infty}\frac{\Gamma(b)\Gamma(c)}{\Gamma(j+b+c)\, j!}\, S_{j}\!\left(x^{2};\frac{b+c-a}{2},\frac{a+c-b}{2},\frac{a+b-c}{2}\right)\!.
\end{eqnarray*}
More conveniently, one can rewrite the expression in terms of the
Beta function,
\[
h(x)=\sum_{j=0}^{\infty}\frac{\mbox{B}(b,j+c)}{(c)_{j}\, j!}\, S_{j}\!\left(x^{2};\frac{b+c-a}{2},\frac{a+c-b}{2},\frac{a+b-c}{2}\right),
\]
and so
\[
h(x)=\int_{0}^{1}(1-t)^{-1+b}\, t^{-1+c}\!\left(\sum_{j=0}^{\infty}\frac{t^{j}}{(c)_{j}\, j!}\, S_{j}\!\left(x^{2};\frac{b+c-a}{2},\frac{a+c-b}{2},\frac{a+b-c}{2}\right)\!\right)\!\mbox{d}t.
\]
Making use of the generating function (see Eq. 9.3.12 in \cite{KoekoekLeskySwarttouw})
\[
\sum_{n=0}^{\infty}\frac{t^{n}}{(\alpha+\beta)_{n}n!}\, S_{n}(x^{2};\alpha,\beta,\gamma)=(1-t)^{-\gamma+ix}\,_{2}F_{1}(\alpha+ix,\beta+ix;\alpha+\beta;t)
\]
one has
\[
h(x)=\int_{0}^{1}(1-t)^{-1+(b+c-a)/2+ix}t^{-1+c}\,_{2}F_{1}\!\left(\frac{b+c-a}{2}+ix,\frac{a+c-b}{2}+ix;c;t\right)\!\mbox{d}t.
\]
Hence
\begin{eqnarray*}
h(x) & = & \sum_{n=0}^{\infty}\frac{1}{(c)_{n}\, n!}\left(\frac{b+c-a}{2}+ix\right)_{\! n}\left(\frac{a+c-b}{2}+ix\right)_{\! n}\\
 &  & \times\,\int_{0}^{1}(1-t)^{-1+(b+c-a)/2+ix}t^{-1+n+c}\,\mbox{d}t\\
 & = & \mbox{B}\!\left(\frac{b+c-a}{2}+ix,c\right)\\
 &  & \times\,\,_{2}F_{1}\!\left(\frac{b+c-a}{2}+ix,\frac{a+c-b}{2}+ix;\frac{b+c-a}{2}+c+ix;1\right)\!.
\end{eqnarray*}
Recalling that \cite[Eq. 15.1.20]{AbramowitzStegun}
\[
\,_{2}F_{1}(\alpha,\beta;\gamma;1)=\frac{\Gamma(\gamma)\Gamma(\gamma-\alpha-\beta)}{\Gamma(\gamma-\alpha)\Gamma(\gamma-\beta)}\ \text{ }\text{if}\ \Re(\gamma-\alpha-\beta)>0
\]
we finally obtain the desired expression. \end{proof}

\subsection{The spectrum of $B(a,b,c)$}

The function $h(x)$ defined in (\ref{eq:h}) is bounded on the positive
half-line and therefore the matrix $B(a,b,c)$ represents a bounded
operator on $\ell^{2}(\mathbb{Z}_{+})$. In what follows, $B(a,b,c)$
is interpreted in this manner.

\begin{corollary} \label{thm:spec_abscont} The absolutely continuous
part of the spectrum of $B(a,b,c)$ is simple and fills the interval
$[0,M(a,b,c)]$ where
\[
M(a,b,c)=\frac{1}{\Gamma(b+c-a)}\,\Gamma\!\left(\frac{b+c-a}{2}\right)^{\!2}\!.
\]
\end{corollary}

\begin{proof} From Theorem~\ref{thm:B_multop} one infers that the
absolutely continuous spectrum of $B(a,b,c)$ fills the closure of
$h([0,+\infty))$. We have $h(0)=M(a,b,c)$ and $h(+\infty)=0$. Moreover,
for any $u$ real fixed, the function $f(x)=\left|\Gamma(u+ix)\right|^{2}$
is monotone decreasing on $(0,+\infty)$. This is immediately seen
from the product formula for the Gamma function \cite[Eq. 6.1.3]{AbramowitzStegun}
yielding
\[
\frac{1}{\left|\Gamma(u+ix)\right|^{2}}=\left(u^{2}+x^{2}\right)e^{2\gamma u}\prod_{n=1}^{\infty}\left[\left(\left(1+\frac{u}{n}\right)^{2}+\left(\frac{x}{n}\right)^{2}\right)e^{-2u/n}\right]
\]
where $\gamma$ is Euler's constant. The assertion follows. \end{proof}

Suppose (\ref{eq:ineq_dicrete}) and recall the notation introduced
in (\ref{eq:N}), (\ref{eq:calM}). Then the point spectrum of $T$
consists of the points $\lambda_{k}^{\,2}$, $k=0,1,\ldots,N(a,b,c)$,
each eigenvalue is simple and for an eigenvector corresponding to
$\lambda_{k}^{\,2}$ one can choose
\[
v_{k}=\sum_{j=0}^{\infty}\hat{P}_{j}(\lambda_{k}^{\,2})\, e_{j}.
\]
These vectors are also eigenvectors of $B$.

\begin{corollary} \label{thm:spec_point} Assuming (\ref{eq:ineq})
and that $a+c-b$, $a+b-c$ are both nonnegative, the point spectrum
of $B(a,b,c)$ is empty. Assuming (\ref{eq:ineq_dicrete}), the point
spectrum of $B(a,b,c)$ equals
\[
\spec_{p}B(a,b,c)=\{\beta_{0},\beta_{1},\ldots,\beta_{N(a,b,c)}\}
\]
 where
\[
\beta_{k}:=h(\lambda_{k})=\frac{\Gamma(b+k)\Gamma(c-a-k)}{\Gamma(b+c-a)}\,,
\]
and it holds true that
\begin{equation}
\beta_{0}>\beta_{1}>\ldots>\beta_{N(a,b,c)}>M(a,b,c).\label{eq:beta_order}
\end{equation}
In particular, all eigenvalues are simple. For an eigenvector corresponding
to $\beta_{k}$ one can choose the vector $v_{k}$ with the components

\[
\left\langle e_{n},v_{k}\right\rangle =\hat{P}_{n}(\lambda_{k}^{\,2})=\sqrt{\frac{(b)_{n}(c)_{n}}{(a)_{n}\, n!}}\,\,_{3}F_{2}(-n,b+k,c-a-k;b,c;1).
\]
Its norm fulfills
\begin{equation}
\|v_{k}\|^{2}=\frac{\Gamma(c)\Gamma(c-a-b-k+1)\, k!}{(c-a-b-2k)\Gamma(c-a-k)\Gamma(c-b-k)(a)_{k}(b)_{k}}\,.\label{eq:vnorm}
\end{equation}
\end{corollary}

\begin{proof} Theorem~\ref{thm:B_multop} implies that the eigenvalues
of $B(a,b,c)$ are exactly the values $h(\lambda_{k})$. $k=0,1,\ldots,N(a,b,c)$.
It remains to show (\ref{eq:beta_order}), (\ref{eq:vnorm}). As far
as (\ref{eq:beta_order}) is concerned, one has
\[
\Gamma(b+c-a)\beta_{k}=\Gamma\!\left(\frac{b+c-a}{2}+\!\left(\frac{a+b-c}{2}+k\right)\right)\!\Gamma\!\left(\frac{b+c-a}{2}-\!\left(\frac{a+b-c}{2}+k\right)\right)
\]
(recall that $a+b-c+2k<0$ for $0\leq k\leq N(a,b,c)$). Hence it
suffices to observe that the function $f(t)=\log(\Gamma(t))$ is convex
on the positive half-line. It is so because
\[
f''(t)=\psi'(t)=\sum_{k=0}^{\infty}\frac{1}{(k+t)^{2}}
\]
where $\psi$ is the digamma function. It follows that, for any $u>0$,
the function $g(t)=\Gamma(u-t)\Gamma(u+t)$ is strictly increasing
on the interval $[0,u)$.

Concerning (\ref{eq:vnorm}), let $\ensuremath{f_{k}\in L^{2}(\mathcal{M}(a,b,c),\mbox{d}\mu(x))}$
be defined as follows: $f_{k}(\lambda_{j})=\delta_{k,j}$ and $f_{k}(x)\equiv0$
on $(0,\infty)$. Then $Uv_{k}=cf_{k}$ for some $c\in\mathbb{C}$.
One immediately finds that
\[
c=(Uv_{k})(\lambda_{k})=\sum_{j=0}^{\infty}\hat{P}_{j}(\lambda_{k}^{\,2})^{2}=\|v_{k}\|^{2}>0.
\]
On the other hand,
\[
\|v_{k}\|^{2}=\|Uv_{k}\|^{2}=c^{2}\|f_{k}\|^{2}=c^{2}\mu(\{\lambda_{k}\})=\|v_{k}\|^{4}\,\mu(\{\lambda_{k}\}).
\]
Whence
\begin{eqnarray*}
\frac{1}{\|v_{k}\|^{2}} & = & \mu(\{\lambda_{k}\})\\
 & = & \frac{(-1)^{k}\Gamma(c-a)\Gamma(c-b)}{\Gamma(c)\Gamma(c-a-b)}\left(1+\frac{2k}{a+b-c}\right)\frac{(a+b-c)_{k}(a)_{k}(b)_{k}}{(a-c+1)_{k}(b-c+1)_{k}\, k!}\\
\noalign{\medskip} & = & \frac{(c-a-b-2k)\Gamma(c-a-k)\Gamma(c-b-k)(a)_{k}(b)_{k}}{\Gamma(c)\Gamma(c-a-b-k+1)\, k!}.
\end{eqnarray*}
This shows (\ref{eq:vnorm}). \end{proof}

\begin{remark} Still assuming (\ref{eq:ineq}), $B(a,b,c)$ is a
positive bounded operator on $\ell^{2}(\mathbb{Z}_{+})$ and one has
\begin{equation}
\left\Vert B(a,b,c)\right\Vert =\frac{1}{\Gamma(b+c-a)}\,\Gamma\!\left(\frac{b+c-a}{2}\right)^{\!2}\label{eq:normB1}
\end{equation}
if $a+b-c\geq0$, $a+c-b\geq0$ (at least one of the expressions is
necessarily positive) and
\begin{equation}
\left\Vert B(a,b,c)\right\Vert =\frac{\Gamma(b)\Gamma(c-a)}{\Gamma(b+c-a)}\label{eq:normB2}
\end{equation}
if $a+b-c<0$, $a+c-b\geq0$, and similarly if $a+b-c\geq0$, $a+c-b<0$.
This means that for every square summable real sequence $\{\xi_{k}\}$,
\[
0\leq\sum_{j=0}^{\infty}\sum_{k=0}^{\infty}B(a,b,c)_{j,k}\xi_{j}\xi_{k}\leq\left\Vert B(a,b,c)\right\Vert \left(\sum_{k=0}^{\infty}\xi_{k}^{\,2}\right)\!,
\]
and the bound is best possible. Equivalently one can also say that
for any real sequence $\{\xi_{k}\}$ and all $n\in\mathbb{Z}_{+}$,
\[
0\leq\sum_{j=0}^{n}\sum_{k=0}^{n}\frac{\Gamma(j+k+a)}{\Gamma(j+k+b+c)}\,\xi_{j}\xi_{k}\leq\left\Vert B(a,b,c)\right\Vert \left(\sum_{k=0}^{n}\frac{\Gamma(k+a)\, k!}{\Gamma(k+b)\Gamma(k+c)}\,\xi_{k}{}^{2}\right)\!,
\]
with $\left\Vert B(a,b,c)\right\Vert $ being specified in (\ref{eq:normB1}),
(\ref{eq:normB2}).

The Hilbert double series inequality is a particular case for $a=\theta$,
$b=\theta$ and $c=1$ assuming that $\theta\geq1/2$. Explicitly,
for any real square summable sequence $\left\{ \xi_{k}\right\} $,
\[
0\leq\sum_{j=0}^{\infty}\sum_{k=0}^{\infty}\frac{\xi_{j}\xi_{k}}{j+k+\theta}\leq\pi\!\left(\sum_{k=0}^{\infty}\xi_{k}^{\,2}\right)^{\!2}\!.
\]
For $a=\theta$, $b=\theta$, $c=1$ and $0<\theta<1/2$ one gets
the inequality
\[
0\leq\sum_{j=0}^{\infty}\sum_{k=0}^{\infty}\frac{\xi_{j}\xi_{k}}{j+k+\theta}\leq\frac{\pi}{\sin(\pi\theta)}\left(\sum_{k=0}^{\infty}\xi_{k}^{\,2}\right)^{\!2}\!.
\]
Again, all bounds are best possible. \end{remark}

\section{Hilbert's matrix and the Bergman-Hilbert matrix}

\subsection{Hilbert's matrix}

As already remarked above, $H(\theta):=B(\theta,\theta,1)$ is the
generalized Hilbert matrix,
\begin{equation}
H(\theta)_{j,k}=\frac{1}{j+k+\theta},\ j,k=0,1,2,\ldots.\label{eq:H}
\end{equation}
By our assumptions on the parameters, $\theta$ is positive. By Corollaries
\ref{thm:spec_abscont} and \ref{thm:spec_point}, the absolutely
continuous part of the spectrum is simple filling the interval $[0,\pi]$
independently of $\theta$. The point spectrum is nonempty if and
only if $0<\theta<1/2$ and if so it consists of the single simple
eigenvalue $\beta_{0}=\pi/\sin(\pi\theta)$.

Observing, however, that the defining expression for $H(\theta)$
is free of square roots, the range of $\theta$ can naturally be extended
to $\theta\in\mathbb{R}\backslash(-\mathbb{Z}_{+})$. The diagonalization
method, as exposed in Section~\ref{sec:diagonalization}, can be
applied to $H(\theta)$ without essential modifications even with
this extended range. This is why we confine ourselves just to sketching
some basic steps.

First of all, $H(\theta)$ commutes with the Jacobi matrix $T(\theta)$
with the entries
\[
T(\theta)_{j,j}=2j\,(j+\theta)-1/4+\theta,\ T(\theta)_{j,j+1}=T(\theta)_{j+1,j}=-(j+1)(j+\theta),\ j=0,1,2,\ldots,
\]
and $T(\theta)_{j,k}=0$ otherwise. Referring to (\ref{eq:ONpoly}),
the associated normalized orthogonal polynomials are given by
\begin{eqnarray*}
\hat{P}_{n}(x^{2}) & = & \frac{1}{n!\,(\theta)_{n}}\, S_{n}\!\left(x^{2};-\frac{1}{2}+\theta,\frac{1}{2},\frac{1}{2}\right)\\
 & = & \frac{(\theta)_{n}}{n!}\,\,_{3}F_{2}\!\left(-n,-\frac{1}{2}+\theta+ix,-\frac{1}{2}+\theta-ix;\theta,\theta;1\right)\!.
\end{eqnarray*}
It is useful to observe that the polynomials $\hat{P}_{n}(x^{2})$
can also be expressed in terms of the Wilson polynomials \cite{Wilson}.
By definition, for $n\in\mathbb{Z}_{+}$,
\[
\frac{W_{n}(x^{2};\alpha,\beta,\gamma,\delta)}{(\alpha+\beta)_{n}(\alpha+\gamma)_{n}(\alpha+\delta)_{n}}=\,_{4}F_{3}(-n,n+\alpha+\beta+\gamma+\delta-1,\alpha+ix,\alpha-ix;\alpha+\beta,\alpha+\gamma,\alpha+\delta;1).
\]
By inspection of \cite[Eqs.~9.1.4, 9.1.5]{KoekoekLeskySwarttouw}
one finds that
\begin{equation}
\hat{P}_{n}(x^{2})=\frac{4^{n}}{n!\,(\theta)_{2n}}\, W_{n}\!\left(\frac{x^{2}}{4};-\frac{1}{4}+\frac{\theta}{2},\frac{1}{4},\frac{1}{4}+\frac{\theta}{2},\frac{3}{4}\right)\!.\label{eq:Phat_Wilson}
\end{equation}

According to (\ref{eq:OGrel}), if $\theta<1/2$ then the orthogonality
relation reads
\begin{equation}
\int_{0}^{\infty}\hat{P}_{m}(x^{2})\hat{P}_{n}(x^{2})\rho(x)\,\mbox{d}x+\sum_{k=0}^{N(\theta)}\mu(\{\lambda_{k}\})\,\hat{P}_{m}(\lambda_{k}^{\,2})\hat{P}_{n}(\lambda_{k}^{\,2})=\delta_{m,n}\label{eq:OG_rel_discr2}
\end{equation}
where $N(\theta)=\left\lceil -1/2-\theta\right\rceil $,
\[
\rho(x)=\frac{2x\tanh(\pi x)}{\Gamma(\theta)^{2}}\left|\Gamma\!\left(-\frac{1}{2}+\theta+ix\right)\right|^{2}
\]
and
\[
\lambda_{k}=i\left(-\frac{1}{2}+\theta+k\right)\!,\ \mu(\{\lambda_{k}\})=\frac{\Gamma(1-\theta)^{2}\,(1-2\theta-2k)}{k!\,\Gamma(2-2\theta-k)}.
\]
The sum on the LHS of (\ref{eq:OG_rel_discr2}) is absent if $\theta\geq1/2$.

\begin{remark} Strictly speaking, the orthogonality relation, as
described for instance in \cite[Eqs.~9.3.2, 9.3.3]{KoekoekLeskySwarttouw},
covers only the cases when $\theta>0$. Nevertheless, making use of
(\ref{eq:Phat_Wilson}) and a very general complex orthogonality relation
for the Wilson polynomials stated in \cite{Wilson}, one can quite
straightforwardly extend the desired formula to all values $\theta<1/2$,
$-\theta\notin\mathbb{Z}_{+}$. Given $\alpha,\beta,\gamma,\delta\in\mathbb{C}$,
write for short $\ensuremath{W_{n}(z)\equiv W_{n}(z;\alpha,\beta,\gamma,\delta)}$,
$n\in\mathbb{Z}_{+}$. As proved in \cite{Wilson},
\[
\frac{1}{2\pi i}\int_{C}f(z)f(-z)W_{m}(z^{2})W_{n}(z^{2})\,\mbox{d}z=\delta_{m,n}Mh_{n}
\]
where
\begin{eqnarray*}
f(z) & = & \frac{\Gamma(\alpha-z)\Gamma(\beta-z)\Gamma(\gamma-z)\Gamma(\delta-z)}{\Gamma(-2z)},\\
M & = & \frac{2\,\Gamma(\alpha+\beta)\Gamma(\alpha+\gamma)\Gamma(\alpha+\delta)\Gamma(\beta+\gamma)\Gamma(\beta+\delta)\Gamma(\gamma+\delta)}{\Gamma(\alpha+\beta+\gamma+\delta)},
\end{eqnarray*}
and
\[
h_{n}=\frac{n!\,(\alpha+\beta+\gamma+\delta-1)_{n}(\alpha+\beta)_{n}(\alpha+\gamma)_{n}(\alpha+\delta)_{n}(\beta+\gamma)_{n}(\beta+\delta)_{n}(\gamma+\delta)_{n}}{(\alpha+\beta+\gamma+\delta)_{2n}}.
\]
The contour $C$ is the imaginary axis deformed so as to separate
the set of poles of $f(z)$ from the set of poles of $f(-z)$, assuming
these two sets to be disjoint. In particular, if $\alpha$, $\beta$,
$\gamma$ and $\delta$ are positive, $C$ may be taken to be the
imaginary axis. The orthogonality measure for the Wilson polynomials
is then positive and supported on the positive real half-line. In
Section~3 of \cite{Wilson} this result is extended to the case when
$\alpha$ is negative while $\alpha+\beta$, $\alpha+\gamma$ and
$\alpha+\delta$ are positive. The poles at $z=\pm(\alpha+k)$, with
$k\in\mathbb{Z}_{+}$ and $\alpha+k<0$, then give rise to mass points
of the orthogonality measure which are located on the negative real
half-line. This procedure can readily be adapted to our case, with
$\alpha=-1/4+\theta/2$, $\beta=1/4$, $\gamma=1/4+\theta/2$ and
$\delta=3/4$, finally resulting in the orthogonality relation (\ref{eq:OG_rel_discr2}).
\end{remark}

Relying on (\ref{eq:OG_rel_discr2}) one can show, similarly as in
Theorem~\ref{thm:B_multop}, that $H(\theta)$ is unitarily equivalent
to the multiplication operator by the function
\[
h(x)=\frac{\pi}{\cosh(\pi x)}
\]
acting on $L^{2}(\mathcal{M}(\theta),\mbox{d}\mu)$ where $\mathcal{M}(\theta)=(0,+\infty)\cup\{\lambda_{k};\, k=0,1,\ldots,N(\theta)\}$
and $\mbox{d}\mu(x)=\rho(x)\mbox{d}x$ on $(0,+\infty)$. The discrete
part of $\mathcal{M}(\theta)$ occurs if and only if $\theta<1/2$.
The corresponding unitary mapping has an analogous form as that in
(\ref{eq:U}). From this explicit diagonalization one immediately
deduces the full information about the spectral properties of $H(\theta)$
thus reproducing the original result due to Rosenblum as stated in
\cite[Theorem 5]{RosenblumII}. Rosenblum's approach was quite different
than ours though an appropriate symmetry was heavily employed, too.
Namely, it has been shown that $H(\theta)$ is unitarily equivalent
to an integral operator on the positive half-line such that there
exists an explicitly diagonalizable Sturm-Liouville operator in its
commutant.

\begin{theorem} \label{thm:H} For all real $\theta$, $\theta\neq0,-1,-2,\ldots$,
the singular continuous part of the spectrum of Hilbert's matrix $H(\theta)$
is empty and the absolutely continuous part is simple and fills the
interval $[0,\pi]$. For $\theta\geq1/2$, the point spectrum of $H(\theta)$
is empty. For $\theta<1/2$, let $N(\theta)=\lceil-1/2-\theta\rceil$.
Then the only possible eigenvalues of $H(\theta)$ are $\pi/\sin(\pi\theta)$
and $-\pi/\sin(\pi\theta)$ whose multiplicities are respectively
equal to $N(\theta)/2+1$ and $N(\theta)/2$ for $N(\theta)$ even,
and they are both equal to $(N(\theta)+1)/2$ for $N(\theta)$ odd.
\end{theorem}

\subsection{The Bergman-Hilbert matrix}

As another application of the general results stated in Theorem~\ref{thm:B_multop}
and Corollaries~\ref{thm:spec_abscont} and \ref{thm:spec_point}
let us consider the so called Bergman-Hilbert matrix $A$ with the
entries
\[
A_{j,k}=\frac{\sqrt{(j+1)(k+1)}}{(j+k+1)^{2}},\ j,k\in\mathbb{Z}_{+},
\]
which has been introduced and studied as an operator on $\ell^{2}(\mathbb{Z}_{+})$
in \cite{Ghatage,DavisGhatage}. It is shown in \cite[Prop.~2]{DavisGhatage}
that the essential spectrum of $A$ equals the interval $[0,1]$.
We can make this analysis more complete by identifying the absolutely
continuous spectrum of $A$.

\begin{proposition} The absolutely continuous spectrum of the Bergman-Hilbert
matrix $A$, regarded as an operator on $\mathbb{\ell}^{2}(\mathbb{Z}_{+})$,
is simple and fills the interval $[0,1]$. \end{proposition}

\begin{proof} Referring to (\ref{eq:Babc}), let $B=B(1,1,2)$. Then
\[
B_{j,k}=\frac{\sqrt{(j+1)(k+1)}}{(j+k+1)(j+k+2)},\ j,k\in\mathbb{Z}_{+},
\]
From Corollary~\ref{thm:spec_abscont} we know that the absolutely
continuous spectrum of $B$ is simple filling the interval $[0,1]$. 

Let us show that $Z:=A-B$ is a trace class operator. We have
\[
Z_{j,k}=\frac{\sqrt{(j+1)(k+1)}}{(j+k+1)^{2}(j+k+2)}.
\]
Expanding
\[
\frac{1}{(j+k+1)^{2}}=\sum_{s=0}^{\infty}\frac{(s+1)\, j^{s}k^{s}}{(j+1)^{s+2}(k+1)^{s+2}}
\]
we can write
\[
Z=\sum_{s=0}^{\infty}T_{s},\ \ \text{where}\ \ T_{s}=(s+1)\, J_{s}H(1)J_{s},\ J_{s}=\diag\!\left\{ \frac{j^{s}}{(j+1)^{s+3/2}};\, j\in\mathbb{Z}_{+}\right\} \!,
\]
where we have used the notation (\ref{eq:H}). From Theorem~\ref{thm:H}
we know that $H(1)$ is positive and so is $T_{s}$. Consequently
($\|\cdot\|_{1}$ standing for the trace norm),
\[
\sum_{s=0}^{\infty}\|T_{s}\|_{1}=\sum_{s=0}^{\infty}\,\sum_{j=0}^{\infty}(T_{s})_{j,j}=\frac{1}{2}\sum_{j=0}^{\infty}\,\sum_{s=0}^{\infty}\frac{(s+1)\, j^{2s}}{(j+1)^{2s+4}}=\frac{\pi^{2}}{16}\,.
\]
The space of trace class operators is complete and therefore $Z$
is, too, a trace class operator.

To conclude the proof we recall that the absolutely continuous spectrum
is known to be invariant with respect to trace class perturbations.
\end{proof}

\section*{Acknowledgments}

One of the authors (P.\v{S}.) wishes to acknowledge gratefully partial
support from grant No. GA13-11058S of the Czech Science Foundation.


\begin{thebibliography}{10}
\bibitem{AbramowitzStegun} M.~Abramowitz, I.~A.~Stegun: \emph{Handbook
of Mathematical Functions}, (Dover Publications, New York, 1972).

\bibitem{BerezinShubin} F.~A.~Berezin, M.~A.~Shubin: \emph{The
Schrödinger Equation}, (Kluwer Academic Publishers, Dordrecht, 1991).

\bibitem{DavisGhatage} C.~Davis, P.~Ghatage: \textit{ On the spectrum
of the Bergman-Hilbert matrix II}, Canad. Math. Bull. \textbf{33}
(1990) 60-64.

\bibitem{Ghatage} P.~G.~Ghatage:\textit{ On the spectrum of the
Bergman-Hilbert matrix}, Linear Algebra Appl. \textbf{97} (1987) 57-63.

\bibitem{Grunbaum81} F.~A.~Grünbaum:\textit{ Toeplitz matrices
commuting with tridiagonal matrices}, Linear Alg. Appl. \textbf{40}
(1981) 25-36.

\bibitem{Grunbaum} F.~A.~Grünbaum:\textit{ A remark on Hilbert's
matrix}, Linear Alg. Appl. \textbf{43} (1982) 119-124.

\bibitem{Kato58} T.~Kato: \emph{On positive eigenvectors of positive
infinite matrices}, Commun. Pure Appl. Math. \textbf{11} (1958) 573-586.

\bibitem{HardyLittlewoodPolya} G.~H.~Hardy, J.~E.~Littlewood,
G.~Pólya:\textit{ Inequalities}, (Cambridge University Press, London,
1934).

\bibitem{KoekoekLeskySwarttouw} R.~Koekoek, P.~A.~Lesky, R.~F.~Swarttouw:
\emph{Hypergeometric Orthogonal Polynomials and Their $q$-Analogues},
(Springer-Verlag, Berlin, 2010).

\bibitem{Otte} P.~Otte: \emph{Diagonalization of the Hilbert matrix},
ICDESFA 2005 conference, Munich (Germany), 25-30 July 2005; URL:\texttt{}~\\
\texttt{http://homepage.ruhr-uni-bochum.de/Peter.Otte/publications.html}.

\bibitem{Peller} V.~V.~Peller: \textit{ Hankel Operators and Their
Applications}, (Springer-Verlag, New-York, 2003).

\bibitem{RosenblumII} M.~Rosenblum: \emph{On the Hilbert Matrix,
II}, Proc. Amer. Math. Soc. \textbf{9} (1958) 581-585.

\bibitem{Varadarajan} V.~S.~Varadarajan: \emph{Geometry of Quantum
Theory}, 2nd ed., (Springer, Berlin, 1985).

\bibitem{Wilf} H.~S.~Wilf: \emph{Finite Sections of Some Classical
Inequalities}, (Springer-Verlag, New York, 1970).

\bibitem{Wilson} J.~A.~Wilson: \emph{Some hypergeometric orthogonal
polynomials}, SIAM J.~Math. Anal. \textbf{11} (1980) 690-701.\end{thebibliography}
\end{document}